\documentclass[11pt]{article}

\usepackage{mathrsfs}
\usepackage{amsmath}
\usepackage{amsfonts}
\usepackage{amssymb}
\usepackage{xcolor, comment}
\usepackage{amsmath, amsthm, amsfonts, amssymb}
\newcommand{\CC}{{\mathbb C}}

\newcommand{\PP}{{\mathbb P}}

\def \-{\bar}

\newcommand{\ord}{\operatorname{ord}}

\newcommand{\FS}{\mathrm{FS}}

\newcommand{\cB}{\mathcal B}

\newcommand{\id}{\operatorname{id}}

\newcommand{\wt}[1]{\widetilde{#1}}
\newcommand{\wh}[1]{\widehat{#1}}
\newcommand{\OO}{\mathcal O}

\newcommand{\Cont}{\operatorname{Cont}}
\newcommand{\Stab}{\operatorname{Stab}}

\newcommand{\ii}{\sqrt{-1}}

\newtheorem{theorem}{Theorem}[section]
\newtheorem{lemma}[theorem]{Lemma}
\newtheorem{corollary}[theorem]{Corollary}
\newtheorem{proposition}[theorem]{Proposition}
\newtheorem{question}[theorem]{Question}
\newtheorem{definition}[theorem]{Definition}

\newtheorem{remark}[theorem]{Remark}

\date{}

\usepackage[colorlinks=true,linkcolor=blue,citecolor=blue,urlcolor=blue]{hyperref}

\newcommand{\C}{\mathbb C}
\newcommand{\cO}{\mathcal O}
\newcommand{\ddc}{\sqrt{-1}\,\partial\bar\partial}
\newcommand{\A}{A^2}

\newcommand{\Sing}{\operatorname{Sing}}
\newcommand{\Reg}{\operatorname{Reg}}
\newcommand{\Vol}{\operatorname{Vol}}
\newcommand{\Aut}{\operatorname{Aut}}

\newcommand{\Graph}{\operatorname{Graph}}
\newcommand{\pr}{\operatorname{pr}}

\begin{document}

\title{\bf Holomorphic Bergman isometries between bounded domains and the extension of Lu's theorem}

\author{Yuan Yuan \medskip \\
Institute for Theoretical Sciences, Westlake University,\\
Hangzhou 310024, Zhejiang, China\\
\texttt{yuanyuan@westlake.edu.cn}}



 \maketitle

\begin{abstract}
We study holomorphic maps between complex manifolds that preserve the Bergman metric up to a positive constant.
In the equal dimensional case, either the source is a Stein manifold and the target is a bounded domain in $\CC^n$, or
 the source is a bounded pseudoconvex domain and the target is a complex manifold, holomorphic Bergman isometry is a biholomorphism onto the complement of a relatively closed locally pluripolar set, and the constant is necessarily one. 
 These include the bounded domains in $\CC^n$ as special cases. 
Some applications to curvature-preserving maps are obtained. 
We further prove the rigidity result for local holomorphic Bergman isometries into Cartesian product of strongly pseudoconvex domains, with applications to finite analytic and modular correspondences.
\end{abstract}

\section{Introduction}

The study of holomorphic maps preserving Bergman metrics has developed
from two complementary perspectives, both of which motivate the present
paper. The first is the theory of holomorphic isometric embeddings
between bounded symmetric domains. Beginning with Bochner and Calabi's seminal
works on K\"ahler immersions and developed extensively in a series of fundamental works by Mok, this
theory exploits the algebraic and homogeneous structures of bounded
symmetric domains to establish extension, algebraicity, and rigidity
properties of local holomorphic isometries
(cf. \cite{B47, C53, M11, M12, M18}). It is also closely connected with arithmetic
geometry. Clozel and Ullmo studied volume-preserving algebraic
correspondences between Shimura varieties \cite{CU03}, while Mok and
Ng proved the total geodesy of the corresponding germs of
measure-preserving maps from a bounded symmetric domain to its
Cartesian product \cite{MN12}. Further rigidity problems for maps
preserving invariant \((p,p)\)-forms were raised by Mok \cite{M11} and obtained in
\cite{Y17,DY26}. 
Motivated by Mok’s deep and influential work, the first objective of this paper is to show that significant aspects of this rigidity phenomenon remain valid beyond the homogeneous setting, in the broader contexts of general bounded domains and Cartesian products of strongly pseudoconvex domains.

The second perspective originates from Lu's classical uniformization
theorem, which states that a bounded domain whose Bergman metric is complete and has
constant holomorphic sectional curvature is biholomorphic to the unit
ball \cite{L66}. Extensive progress has recently been made toward
removing the completeness assumption. Dong and Wong obtained several
extensions under additional hypotheses by using the Bergman-Calabi
diastasis and Bergman representative coordinates
\cite{DW22a,DW22b}. A decisive breakthrough was achieved by Huang and Li
in \cite{HL25a}, where they established a deep and far-reaching
generalization of Lu's theorem. They proved that, for a bounded
pseudoconvex domain, the correct conclusion in the absence of
completeness is that the domain is biholomorphic to the unit ball with a
relatively closed pluripolar subset removed.

More recently, Huang and Li extended this result to arbitrary Stein
manifolds, including unbounded pseudoconvex domains in complex Euclidean
space \cite{HL26}. Their work identifies the natural geometric defect
that is invisible to the Bergman space and replaces the purely geometric
assumption of Bergman completeness in Lu's theorem with the complex
analytic and geometric condition of pseudoconvexity.

The paper \cite{HL25a} has stimulated a number of subsequent
developments, including
\cite{ETX25a,ETX25b,HK26,LP26,Ya26,BGNX26}, and serves as one of the
principal sources of inspiration for the present work. Ebenfelt, Treuer, and Xiao introduced the notion of
\emph{Bergman-negligible sets} and clarified the corresponding
uniformization theory for general bounded domains
\cite{ETX25b}. Loi and Palmieri established an analogous
uniformization theorem for bounded domains with locally symmetric
Bergman metrics \cite{LP26}. These developments strongly suggest that
local metric or curvature equivalence of the Bergman metric should
determine the global biholomorphic structure of a bounded domain, up to
a thin subset that is undetectable by square-integrable holomorphic
functions.

In \cite{HL26}, Huang and Li further showed that a Stein manifold with a
well-defined Bergman metric of non-positive constant holomorphic
sectional curvature is biholomorphic to the complex ball with a
pluripolar subset removed. In sharp contrast to Lu's theorem, they also
constructed, in every complex dimension, Stein manifolds whose Bergman
metrics have positive constant holomorphic sectional curvature, although
their Bergman spaces do not even separate points. Recall that Lu's
theorem implies that no Stein manifold can admit a \emph{complete}
Bergman metric of positive constant holomorphic sectional curvature.

Our first main result is a mapping-theoretic counterpart of this principle and may be viewed as a further development of the  work of Lu, Huang, and Li.

\begin{theorem}\label{domain}
Let \(D_1,D_2\subset\C^n\) be bounded
domains with $n \geq 1$ and $D_1$ be pseudoconvex.  If
$  f:D_1\rightarrow D_2$
is a holomorphic Bergman isometry satisfying
$ f^*\omega_{D_2}=\lambda\omega_{D_1}$ 
for some $\lambda>0$,
then there exists a relatively closed pluripolar set $E$ in \(D_2\), 
such that
$  f:D_1\rightarrow D_2\setminus E $
is biholomorphic. Moreover, \(\lambda=1\).
\end{theorem}

In fact, we obtain following two results, both of which are the generalization of Theorem \ref{domain} to Stein manifolds.

\begin{theorem}\label{main1}
Let $D$ be a bounded domain in $\mathbb{C}^n$ and $M$ be a Stein manifold of complex dimension $n$ with well-defined Bergman metric $\omega_M$. 
 If
$  f: M\rightarrow D$
is a holomorphic Bergman isometry satisfying
$ f^*\omega_{D}=\lambda\omega_{M}$ 
for some $\lambda>0$,
then there exists a relatively closed pluripolar set $E$ in \(D\), 
such that
$  f: M\rightarrow D \setminus E $
is biholomorphic. Moreover, \(\lambda=1\).
\end{theorem}

\begin{theorem}\label{thm:main2}
Let $D\subset\CC^n$ be a bounded pseudoconvex domain. Let $M$ be a complex manifold of complex dimension
$n$  whose  Bergman-Bochner map is an immersion. 
 If $f:D \to M$ is a holomorphic Bergman isometry satisfying
$ f^*\omega_{M}=\lambda\omega_{D}$ 
for some $\lambda>0$,
then there exists a relatively closed locally pluripolar set $E$ in \(M\), 
such that
$  f: D \rightarrow M \setminus E $
is biholomorphic. Moreover, \(\lambda=1\).
\end{theorem}

The deep extension results due to Mok \cite{M12} and Huang-Li \cite{HL26} yield the local version of above theorems.

 \begin{theorem}\label{al2}
 Let $M$ be a Stein manifold of complex dimension $n$ with well-defined Bergman metric $\omega_M$,
and let $D \subset \CC^n$ be a bounded pseudoconvex domain with complete Bergman metric $\omega_D$. Suppose there exists a connected open set $U$ in $M$ and a holomorphic map $f: U \to D$ such that $f^*\omega_D =\lambda  \omega_M$ on $U$ for some $\lambda>0$.
Then there exists a relatively closed pluripolar set $E$ in $D$ such that 
$f$ extends to a biholomorphism $F : M \to D\setminus E$. 
 \end{theorem}

  \begin{theorem}\label{al3}
  Let $M$ be a Stein manifold of complex dimension $n$ with well-defined complete Bergman metric $\omega_M$ and the Bergman space $\A(M)$ separates points on $M$.
Let $D \subset \CC^n$ be a bounded pseudoconvex domain. Suppose there exists a connected open set $U$ in $D$ and a holomorphic map $f: U \to M$ such that $f^*\omega_M = \lambda \omega_D$ on $U$ for some $\lambda>0$.
Then there exists a relatively closed pluripolar set $E$ in $M$ such that 
$f$ extends to a biholomorphism $F : D \to M\setminus E$.  
 \end{theorem}

The conclusion is optimal, as shown by the
non-surjective Bergman self-isometries constructed in Section~3.6.
The proof partially  follows from the novel ideas of 
 Huang and
Li in \cite{HL25a}. Nevertheless, by a local argument for the uniformization theorem, in their case, the domain is locally biholomorphically isometric to the complex unit ball, whose Bergman kernel function has nice properties and also there is no monodromy when extending the Bergman-Bochner map of the complex unit ball. In this paper, 
Calabi's extension theorem, in the form
developed by Huang and Li \cite{HL25b}, yields continuation of its germ
along every path in \(D_2\). 
The argument has to 
resolve the possible monodromy on a covering Riemann domain and
combine the Bergman-Bochner map with Irgens' \(L^2\)-envelope of
holomorphy and Josefson's theorem \cite{I04,J78}.
Moreover, upon dropping the pseudoconvexity assumption and imposing the normalization \(\lambda=1\), one can obtain certain partial results by applying the powerful moment theory  (cf. \cite{S17}). This setting is included in the more general framework considered in \cite{ETX25b}.

By the classical theorem of Nomizu and Yano \cite{NY67}, a strongly
curvature-preserving local biholomorphism between locally irreducible
Bergman metrics is a local holomorphic Bergman isometry. Then Theorem \ref{domain} and its local version yield the rigidity result for such strongly
curvature-preserving local biholomorphisms.
These results may be compared with generalizations of Lu's
theorem and of the recent theorem of Loi-Palmieri, which begin with an intrinsic curvature condition and identify
the resulting global model. Our results instead begin with a local
curvature equivalence between two arbitrary bounded domains and show
that this equivalence is globally induced by a biholomorphism, modulo a
pluripolar defect in the pseudoconvex case and without any defect in
the complete case. When the target is the unit ball, this recovers the
spirit of Lu's uniformization theorem. When a local curvature
equivalence with an irreducible bounded symmetric domain is given, it
provides the corresponding locally symmetric picture of
Loi-Palmieri. More generally, neither Bergman metric is required to
be locally symmetric: agreement of the full curvature jets along the
given map is sufficient.

The first perspective reappears in our Cartesian product theorem. 

\begin{theorem}\label{product}
Let \(D\Subset\C^n\), \(n\geq2\), be a simply connected, bounded,
strongly pseudoconvex domain with real analytic boundary and $U \subset D$ be a connected open subset.  For $1\leq j\leq m$, let
$  f_j: U\rightarrow D$
be a holomorphic map, and assume that every \(f_j\) has full rank
\(n\) at some point of \(U\).  If there exists $\lambda>0$ such that
\begin{equation*}
  \sum_{j=1}^{m}f_j^*\omega_D=\lambda\omega_D
\end{equation*}
holds on $U$, 
then each $  f_j$ extends to an element in $\Aut(D)$ and thus $ \lambda=m.$
\end{theorem}

 Inspired by the work of Mok \cite{M02}, Clozel-Ullmo \cite{CU03} and 
Mok-Ng \cite{MN12} for bounded symmetric domains, this result extends
componentwise rigidity to a broad nonhomogeneous class. Its proof
combines Mok's extension theorem with boundary analytic continuation \cite{M12}
and the uniformization theorem of Nemirovski and Shafikov
\cite{NS05}. Following the framework of Clozel-Ullmo, we obtain
applications to finite analytic correspondences and show that
metric-preserving correspondences on regular quotients are modular.

The present paper mainly treats the equidimensional problem. 
The positive-codimensional problem will be addressed in a forthcoming
work.

\section{Preliminaries}

For a bounded domain \(G\subset\C^n\), let $  \A(G)=L^2(G)\cap\cO(G)$ be the Bergman space.
  Its reproducing kernel is denoted
by \(K_G(z,w)\). 
 The fundamental form of the Bergman metric is denoted by
$ \omega_G=\ddc\log K_G(z,z)$ (which is also called Bergman metric). 
It is positive definite on every bounded domain.  Indeed, constants and
affine-linear functions belong to \(\A(G)\).  Given \(z\in G\) and
\(0\ne v\in T_zG\), choose a complex-linear functional \(L\) with
\(L(v)\ne0\) and set \(\ell(w)=L(w-z)\).  Then
\(\ell(z)=0\) and \(d\ell_z(v)\ne0\), whereas a nonzero constant has
zero derivative.  Hence the derivative of the projective evaluation
map cannot vanish in the direction \(v\).  Thus the Bergman-Bochner
map below is an immersion, and the pullback of the Fubini-Study form
is positive definite.

If \((\phi_j)_{j\geq 0}\) is an orthonormal basis of \(\A(G)\), then
\[ K_G(z,w)=\sum_{j=0}^{\infty}  \phi_j(z)\overline{\phi_j(w)}\]
locally uniformly on \(G\times G\).  In particular,
\[  K_G(z,z)=\sum_{j=0}^{\infty}|\phi_j(z)|^2.\]
Write \(\PP^\infty:=\PP(\ell^2)\), the projective Hilbert space.  It is
a Hausdorff complex Hilbert manifold; its holomorphic maps have local
holomorphic \(\ell^2\setminus\{0\}\)-valued lifts, unique up to a
nowhere-vanishing holomorphic scalar.  Consequently, the
Bergman-Bochner map
\[  \cB_G:G\rightarrow\PP^\infty,\quad \cB_G(z)=[\phi_0(z):\phi_1(z):\cdots] \]
is holomorphic and satisfies
\begin{equation}\label{eq:BB-pullback}
  \cB_G^*\omega_{\FS}
  =\ddc\log\sum_{j=0}^{\infty}|\phi_j|^2
  =\omega_G.
\end{equation}
The compact-local convergence of the kernel series also says that
\(z\mapsto(\phi_j(z))_{j\geq0}\) is a holomorphic \(\ell^2\)-valued map
in the sense used for the Hilbert projective space \(\PP^\infty\).
It follows that is $\cB_G$ is a holomorphic isometry from
\((\Omega,\omega_\Omega)\) into
\((\PP^\infty,\omega_{\FS})\).

Let $M$ be a complex manifold of complex dimension $n$. Throughout this paper, all complex manifolds are assumed to be connected. 
We denote by
$$    \A(M)  =   \left\{  \varphi\in \OO_{(n, 0)}(M):  \ii^{\,n^{2}}\int_M \varphi\wedge\overline{\varphi}<\infty \right\} $$
the Bergman space of square-integrable holomorphic $(n, 0)$-forms.  
We say that $\A(M)$ separates points if, whenever $p,q\in M$ and
$p\neq q$, there exists $\varphi\in\A(M)$ such that
$  \varphi(p)=0,    \varphi(q)\neq 0.$
The Bergman metric $\omega_M$ and the Bergman-Bochner map $\cB_M: M \to \PP^\infty$ can be defined similarly. Note that 
 $\A(M)$ is base point free if and only if $\cB_M$ is defined everywhere on $M$;
$\A(M)$ separates points if and only if $\cB_M$ is injective; if $\cB_M$ is an immersion, then the Bergman metric $\omega_M = \cB_M^* \omega_{\FS}$ is well defined.

\begin{definition}
Let $M_1, M_2$ be complex manifolds of complex dimension $n, N$, respectively, with  well defined Bergman metrics $\omega_{M_1}, \omega_{M_2}$. 
A holomorphic map $f: M_1 \rightarrow M_2$ is called {\it holomorphic Bergman isometry} if there exists a positive number $\lambda$ such that  
\begin{equation}\label{bgiso}
f^* \omega_{M_2} = \lambda \omega_{M_1}.
\end{equation}
If $U$ is a connected open subset of $M_1$ and a holomorphic map $f: U \rightarrow M_2$ satisfies (\ref{bgiso}) on $U$, then $f$ is called a local holomorphic Bergman isometry from $M_1$ to $M_2$. 
 $f$ is called bona fide if $\lambda=1$. 
\end{definition}

The injectivity of holomorphic Bergman isometries between bounded domains in the complex Euclidean spaces is obtained in \cite{M12}. 
The proof extends verbatim to general complex manifolds. 
For reader's convenience, we include the detailed proof here. The idea of proof is essentially  from \cite{M12, HL25a}.

\begin{proposition}\label{inj}
Let $M_1, M_2$ be complex manifolds of complex dimension $n, N$, respectively, with  well defined Bergman metrics $\omega_{M_1}, \omega_{M_2}$. 
Suppose that the Bergman space $\A(M_1)$ separates points of $M_1$.  
If $f: M_1 \to M_2$ is a holomorphic Bergman isometry satisfying
$ f^*\omega_{M_2}=\lambda\omega_{M_1}$ 
for some $\lambda>0$, then $f$ is injective.
\end{proposition}

\begin{proof}
Since $f$ is a holomorphic Bergman isometry from $M_1$ to $M_2$, $f$ is a holomorphic immersion.
We shall prove that $f$ is globally one-to-one. Suppose that 
there exist points $a,b\in M_1$ such that
$  f(a)=f(b).$
We will show that this forces $a=b$.

For a complex manifold $M$ with well-defined Bergman metric, recall the Calabi's diastasis centered at
$w \in M$ by
\[    D_M(z,w)      = \log  \frac{K_M(z,z)K_M(w,w)}  {|K_M(z,w)|^2},\]
wherever this expression is defined, namely at those $z \in M$ such that $K_M(z, w)\neq0$. Moreover, for fixed $w \in M$, 
$\partial_z \bar\partial_z   D_M(z,w) = \partial_z \bar\partial_z \log K_M(z,z)$. 
For $z$ near $a$, both $K_{M_1}(z,a)$ and $K_{M_2}(f(z),f(a))$ are nonzero. Hence the functions
$        D_{M_1}(z,a)$ and $        D_{M_2}(f(z),f(a))$
are well-defined near $a$.

Since $f$ is a Bergman isometry, 
we have
\[ \sqrt{-1}\,\partial_z \bar\partial_z D_{M_2}(f(z),f(a))   =  \lambda \sqrt{-1}\,\partial_z \bar\partial_z D_{M_1}(z,a). \]
Thus $        D_{M_2}(f(z),f(a))-    \lambda D_{M_1}(z,a)$
is pluriharmonic near $a$.
By the uniqueness of Calabi's diastasis, both sides have zero value at $z=a$
and have no purely holomorphic or purely antiholomorphic terms in their Taylor
expansions at $a$. 
Hence
\begin{equation}\label{diastasis identity}
        D_{M_2}(f(z),f(a))  =  \lambda D_{M_1}(z,a)
\end{equation}
holds for $z$ sufficiently close to $a$ \cite{C53}.

Define $        H(z):=K_{M_1}(z,a)K_{M_2}(f(z),f(a)).$
Then 
$ H(a)=K_{\Omega_1}(a,a)K_{\Omega_2}(f(a),f(a))\neq0$, and $H$ is holomorphic section of a holomorphic line bundle over $M_1$.  
Let $ Z  = \{z\in M_1:H(z)=0\}$. Then $Z$ is 
 a proper complex analytic subset of $M_1$.
It follows that the complement
$   M_1\setminus Z$ is connected. 
Both sides are real analytic on this connected open set and agree near \(a\); hence
(\ref{diastasis identity}) holds
for all $z\in M_1\setminus\{H=0\}$.

We now claim that $        K_{M_1}(b,a)\neq0.$ 
Suppose instead that $   K_{M_1}(b,a)=0.$
Since
\[ K_{M_2}(f(b),f(a))=K_{M_2}(f(a),f(a))>0, \]
we may choose a sequence $z_\nu\to b$ such that
$   H(z_\nu)\neq0.$
For such $z_\nu$, one has
\[ D_{M_2}(f(z_\nu),f(a)) =  \lambda D_{M_1}(z_\nu,a).\]
As $\nu\to\infty$, the left-hand side tends to
\[ D_{M_2}(f(b),f(a)) = D_{M_2}(f(a),f(a))  = 0.\]
But the right-hand side tends to $+\infty$, because
\[ K_{M_1}(z_\nu,a)\to K_{M_1}(b,a)=0\]
while
\[ K_{M_1}(z_\nu,z_\nu)\to K_{M_1}(b,b)>0,  \quad K_{M_1}(a,a)>0.\]
This contradiction proves that $ K_{M_1}(b,a)\neq0.$

Therefore (\ref{diastasis identity}) holds  at $z=b$. Since
$f(b)=f(a)$, we obtain
\[  0 = D_{M_2}(f(b),f(a))  =  \lambda D_{M_1}(b,a).\]
Because $\lambda>0$, this implies
$  D_{M_1}(b,a)=0.$ 
Thus
\[  \log \frac{K_{M_1}(b,b)K_{M_1}(a,a)}{|K_{M_1}(b,a)|^2} = 0,\]
or equivalently,
\[ |K_{M_1}(b,a)|^2 = K_{M_1}(b,b)K_{M_1}(a,a).\]
This is precisely the equality case in the Cauchy-Schwarz inequality.  
Hence $  K_{M_1}(\cdot,a), K_{M_1}(\cdot,b)$ are
linearly dependent. Namely, there exists a constant $c\in\mathbb C\setminus \{0\}$ such
that $  K_{M_1}(\cdot,b)=cK_{M_1}(\cdot,a).$ 
By the reproducing property, 
$ h(b)=\overline c\, h(a)$ for every $h\in A^2(M_1)$.
Therefore, for every $h \in\A(M_1)$,
if  $  h(a)=0$, then $h(b)=0$. It follows from the point separation property of $\A(M_1)$ that $a=b$. 
Hence $f$ is injective.
\end{proof}

In order to study the local holomorphic Bergman isometries between complex manifolds, we need the following deep results due to Mok and Huang-Li. Although Mok only considered the case of bounded domains in the complex Euclidean space,  
his proof in \cite{M12} extends verbatim to yield the following result.

\begin{proposition}[Mok]\label{m2012}
Let \(M_1,M_2\) be complex manifolds with well-defined Bergman metrics. Assume that the Bergman space \(A^2(M_2)\) separates points in $M_2$ and that \((M_2, \omega_{M_2})\) is complete. Suppose there exists a connected open set $U$ in $M_1$ and a holomorphic map $f: U \to M_2$ such that $f^*\omega_{M_2} =\lambda \omega_{M_1}$ on $U$ for some $\lambda>0$. Then $f$ extends to a holomorphic Bergman isometry $F : M_1 \to M_2$.
\end{proposition}

\begin{proposition}[Proposition 3.1 and Remark 3.4 in \cite{HL26}]\label{hl26}
Let $M$ be a Stein manifold of complex dimension $n$ with well-defined Bergman metric $\omega_M$,
and let $D \subset \CC^n$ be a bounded pseudoconvex domain with complete Bergman metric $\omega_D$. Suppose there exists a connected open set $U$ in $M$ and a holomorphic map $f: U \to D$ such that $f^*\omega_D =\lambda \omega_M$ on $U$ for some $\lambda>0$. 
Then $f$ extends to a holomorphic Bergman isometry $F : M \to D$ and the Bergman space $\A(M)$ separates points of $M$.
\end{proposition} 

The following results follow by combining Proposition \ref{inj}, Proposition \ref{m2012} and Proposition \ref{hl26}. 

\begin{corollary}\label{yy}
Let $M$ be a Stein manifold of complex dimension $n$ with well-defined Bergman metric $\omega_M$,
and let $D \subset \CC^n$ be a bounded pseudoconvex domain with complete Bergman metric $\omega_D$. Suppose there exists a connected open set $U$ in $M$ and a holomorphic map $f: U \to D$ such that $f^*\omega_D = \lambda \omega_M$ on $U$ for some $\lambda>0$.
Then $f$ extends to an injective holomorphic Bergman isometry $F : M \to D$.
\end{corollary}

\begin{corollary}\label{yy2}
Let $M$ be a Stein manifold of complex dimension $n$ with the well-defined complete Bergman metric $\omega_M$ and the Bergman space $\A(M)$ separates points on $M$.
Let $D \subset \CC^n$ be a bounded  domain. Suppose there exists a connected open set $U$ in $D$ and a holomorphic map $f: U \to M$ such that $f^*\omega_M = \lambda \omega_D$ on $U$ for some $\lambda>0$.
Then $f$ extends to an injective holomorphic Bergman isometry $F : D \to M$.
\end{corollary}

\section{The extension of Q.-K. Lu's theorem}
\subsection{On the isometric constant}

\begin{lemma}\label{lem:vanishing}
Let \(G\subset\C^n\) be a bounded domain and \(h\in\OO(G)\). If there exist positive constants \(a, C\) 
such that
\(  |h(z)|^2K_G(z)^a\leq C\) for 
all \(z\in G \), 
 then \(h\equiv0\).
\end{lemma}

\begin{proof}
It follows from the definition of the 
Bergman kernel that  \(  K_G(z)\geq \frac{1}{\Vol(G)}.\) 
Thus \(h\) is bounded. Let \(M:=\sup_G|h|\). We argue by contradiction.

Suppose first that \(h\) is nonconstant and nonzero. Then \(M>0\), and
the maximum principle yields \(|h(z)|<M\) for every \(z\in G\). For each
positive integer \(k\), \(h^k \in \A(G)\),
and it follows from the extremal characterization of the Bergman kernel that
\(  K_G(z)\geq \frac{|h(z)|^{2k}}{\|h^k\|_{L^2(G)}^2}. \)
Consequently,
\[C\geq |h(z)|^2K_G(z)^a\geq  \frac{|h(z)|^{2ak+2}}{\|h^k\|_{L^2(G)}^{2a}}.\]
Taking the supremum over \(z\), we have 
\[ \frac{\|h^k\|_{L^2(G)}^2}{M^{2k}}  \geq \frac{M^{2/a}}{C^{1/a}}>0.\]
On the other hand,
\[  \frac{\|h^k\|_{L^2(G)}^2}{M^{2k}} = \int_G\left(\frac{|h(z)|}{M}\right)^{2k}dV(z) \rightarrow0 \]
by the dominated convergence theorem. This is a contradiction.

Now suppose that \(h\) is a nonzero constant.
Then \(K_G\) is uniformly bounded on $G$ by the assumption. 
This is also impossible since $G$ is a bounded domain.
\end{proof}

\begin{proposition}\label{constant}
Let \(D_1, D_2\subset\C^n\) be bounded domains, and let
\( f: D_1 \rightarrow D_2 \)
be a holomorphic Bergman isometry satisfying
\[  f^*\omega_{D_2}=\lambda\omega_{D_1}\]
for some constant \(\lambda>0\). Then \(\lambda\leq1\). 
\end{proposition}

\begin{proof}
Let \( \Omega=f(D_1).\)
By Proposition \ref{inj}, \(f:D_1\to\Omega\) is biholomorphic. Moreover, 
 $ \omega_{D_2}|_\Omega=\lambda\omega_\Omega.$
 By the similar argument as in the proof of   Proposition \ref{inj}, 
\begin{equation}\label{eq:normalized}
  \frac{|K_{\Omega}(z,w)|^2}{K_{\Omega}(z,z)K_{\Omega}(w,w)}  = \left(  \frac{|K_{D_2}(z, w)|^2}{K_{D_2}(z, z)K_{D_2}(w,w)} \right)^{\frac1\lambda} 
  \end{equation}
holds on 
$  \bigl\{(z, w)\in  \Omega\times\Omega:  K_{\Omega}(z, w)K_{D_2}(z, w)\ne0\bigr\}.$
It then follows from the  continuity that \eqref{eq:normalized} holds at any
\((z,w)\in\Omega\times\Omega\). In particular, for every fixed
\(\xi\in\Omega\), the zero sets of
\(  A_\xi(z):=K_{D_2}(z,\xi)\)
and
 \( B_\xi(z):=K_{\Omega}(z,\xi) \)
have the same support.
Let \(Y\) be an irreducible hypersurface component of this common
support. At a generic smooth point of \(Y\), restricting
\eqref{eq:normalized} to a small complex disk transverse to \(Y\),
then
\begin{equation} \label{eq:orders}
 \ord_Y(B_\xi)=\frac{1}{\lambda}  \,\ord_Y(A_\xi) 
  \end{equation}
holds by comparing the vanishing orders. 

We now prove that \(\lambda\leq1\). Suppose, to the contrary, that
\(\lambda>1\).
Choose positive integers \(p<q\)
such that
\( \varepsilon:=p-\frac{q}{\lambda} \in[0,1). \)
If \(\frac{1}{\lambda}=\frac{r}{s}\) is rational in the simpliest form, take \(p=r\) and \(q=s\),
so that \(\varepsilon=0\). If \(\lambda\) is irrational, choose \(q\) so
large that \(\lceil \frac{q}{\lambda}\rceil<q\), and put \(p=\lceil \frac{q}{\lambda}\rceil\).
Fix \(a\in\Omega\). The quotient
$  F_a(z):=\frac{A_a(z)^p}{B_a(z)^q}$
is a priori meromorphic. For every irreducible hypersurface
component \(Y\) of the common zero support, \eqref{eq:orders} yields
\[  \ord_Y(F_a)=  p\,\ord_Y(A_a)-q\,\ord_Y(B_a)=  \varepsilon\,\ord_Y(A_a)\geq0.\]
Thus \(F_a\) has no pole divisor 
and then \(F_a\) extends holomorphically to $\Omega$ across the
zero set of \(B_a\). Notice that when \(\lambda\) is irrational,
\eqref{eq:orders} forces the common zero support to be empty, whereas
in the rational case our choice yields \(\varepsilon=0\).
On the complement of the common zero set, taking the \(q\)-th power of
\eqref{eq:normalized} with \(w=a\) yields
\begin{equation} \label{eq:F}
  |F_a(z)|^2 =  C_a\,|A_a(z)|^{2\varepsilon}  \frac{K_{D_2}(z, z)^{\frac{q}{\lambda}}}{K_{\Omega}(z,z)^q},
   \quad
  C_a:=\frac{K_{D_2}(a, a)^{\frac{q}{\lambda}}}{K_\Omega(a, a)^q}>0.
\end{equation}
This equality extends everywhere by continuity. When
\(\varepsilon=0\), the factor
\(|A_a(z)|^{2\varepsilon}\) is  identically \(1\),
including at zeros of \(A_a\).
The 
Cauchy--Schwarz inequality yields
\[ |A_a(z)|^2\leq K_{D_2}(z, z)K_{D_2}(a, a).\]
Moreover, since \(\Omega\subset D_2\), 
 the monotonicity inequality of the Bergman kernel yields
\[  K_{D_2}(z, z)\leq K_{\Omega}(z,z)\]
for all \(z\in\Omega\).
Using \(\frac{q}{\lambda}+\varepsilon=p\) in \eqref{eq:F}, and absorbing the fixed
factor \(K_{D_2}(a, a)^\varepsilon\) into the constant, we obtain
\[  |F_a(z)|^2  \leq C'_a\frac{K_{D_2}(z, z)^p}{K_{\Omega}(z,z)^q}  \leq C'_a K_{\Omega}(z,z)^{p-q}.\]
Equivalently,
\[ |F_a(z)|^2K_{\Omega}(z,z)^{q-p}\leq C'_a.\]
Since \(q-p>0\), Lemma~\ref{lem:vanishing}, applied on \(\Omega\),
implies \(F_a\equiv0\). This is impossible because
$$  F_a(a)=\frac{K_{D_2}(a, a)^p}{K_{\Omega}(a, a)^q}\ne0.$$
Therefore \(\lambda\leq1\).
\end{proof}

\subsection{Stein case}
We first recall the definition of the pluripolar set on a complex manifold.
\begin{definition}
Let $M$ be a connected complex manifold and let $E\subset M$.
The set $E$ is called locally pluripolar in $M$ if, for every
$p\in E$, there exist an open neighborhood $U_p\subset M$ of $p$ and
a function $    u_p\in\operatorname{PSH}(U_p)$ with $ u_p\not\equiv -\infty$,
such that $    E\cap U_p\subset\{u_p=-\infty\}.$
The set $E$ is called (globally) pluripolar in $M$ if there
exists a function $    u\in\operatorname{PSH}(M)$ with $ u\not\equiv-\infty,$
such that $    E\subset\{u=-\infty\}.$
\end{definition}

Note that every globally pluripolar set is locally pluripolar. The converse need not hold on an arbitrary complex manifold. It does hold when \(M\) is Stein, by the Stein version of Josefson’s theorem (cf. Remark 1 in \cite{B80} and
the discussion preceding Theorem 6.8 in \cite{BT82}).  

There are two key ingredients in the proof of the main theorem: Calabi's extension theorem (cf. \cite{C53, HL25b}) and Irgens' continuation theorem for $L^2$-holomorphic functions (cf. \cite{I04, PZ02}). We will apply them to prove the following characterization for the complement. For the detailed account of Riemann domain and the pluripotential theory, the reader is referred to the classical textbooks (cf. \cite{GF76, GR65, K91}).

\begin{proposition}\label{lem:continuation}
Let \(n\geq1\), and let $  \Omega\subset G\subset\C^n,$ 
where \(\Omega\) is a bounded pseudoconvex domain and \(G\) is a
connected domain.  Fix \(z_0\in\Omega\).  Assume that the germ at
\(z_0\) of the
Bergman-Bochner map $\cB$ of \(\Omega\) extends holomorphically
 along every continuous curve in \(G\) starting at
\(z_0\).
Then \(G\setminus\Omega\) is pluripolar.
\end{proposition}

\begin{proof}
Choose an orthonormal basis \((\phi_j)_{j\geq0}\) of \(\A(\Omega)\)
with $  \phi_0\equiv \Vol(\Omega)^{-\frac{1}{2}}.$
The Bergman-Bochner map of $\Omega$ is
\[  \cB=[\phi_0:\phi_1:\cdots]: \Omega \to \PP^\infty.\]
 Write $ b_{0}:=\operatorname{germ}_{z_{0}} \cB$
to be the original germ of $\cB$ at $z_0$. 
Along every continuous curve in \(G\) connecting
\(z_0\) and $z$, by holomorphic continuation, $\cB$ extends as a germ of holomorphic map at $z$ from $G$ to $\PP^\infty$. When two curves are homotopic to each other, by the classical monodromy argument, the germs of $\cB$ at $z$ coincide. Moreover, let $\alpha$ be a loop in $G$ based at $z_0$ and denote the holomorphic continuation of $b_0$ along $\alpha$ by  $\Cont_{\alpha}(b_{0}).$
Let 
\begin{equation}\label{eq:stabilizer}
  H=\Stab(b_{0})  :=\bigl\{[\alpha]\in\pi_{1}(G,z_{0}):  \Cont_{\alpha}(b_{0})=b_{0}\bigr\}.
\end{equation}
Thus $H$ consists of exactly those homotopy classes of loops for which
holomorphic continuation returns to the original branch of $\cB$. It is straightforward to verify that $H$ is a subgroup. 
By the classification of connected coverings, the subgroup
\(H\subset\pi_{1}(G,z_{0})\) determines a connected covering
$ p \colon X\rightarrow G$
with a distinguished point $x_{0}\in p^{-1}(z_{0})$, such that $  p_{*}\pi_{1}(X,x_{0})=H.$
Equivalently, a loop in $G$ based at $z_{0}$ lifts to a loop in $X$
based at $x_{0}$ if and only if its homotopy class belongs to $H$.
A point  $x \in p^{-1} (z)$ of $X$ over $z\in G$ can be represented by a path
\(\gamma\) from  $z_{0}$ to $z$.  The value of the lifted map at $x$ is
defined by continuing $b_{0}$ along \(\gamma\). If two paths represent the
same point of $X$, the loop obtained by following one path and returning
along the other lies in $H$. It follows from \eqref{eq:stabilizer}  that
this loop produces no change of branch.  The continued value is therefore
independent of the representative, and holomorphic continuation becomes a
single-valued holomorphic map $  \widetilde \cB\colon X\rightarrow\PP^{\infty}.$
In this sense, $X$ separates precisely the branches created by monodromy and
it is the natural connected covering with the smallest number of sheets on
which the chosen original branch becomes single-valued. Since
$\c B$ is already a single-valued holomorphic map on \(\Omega\), continuation around every loop lying in \(\Omega\) returns to the
original germ. If $i\colon\Omega\hookrightarrow G$ denotes the inclusion,
then $  i_{*}\bigl(\pi_{1}(\Omega,z_{0})\bigr)\subseteq H.$ 
The  lifting criterion now yields a lift
\[  j\colon\Omega\rightarrow X,
 \quad p\circ j=\operatorname{id}_{\Omega},
  \quad j(z_{0})=x_{0}. \]
Because $p\circ j=\operatorname{id}_{\Omega}$, the map $j$ is injective.
Since $p$ is locally biholomorphic, $j$ is a holomorphic open embedding.
Thus $j(\Omega)$ is the distinguished original sheet of $X$, and
$ \widetilde B\circ j=B.$

Since  \(p:X\to G\) is the connected covering corresponding to
\(H\),
 \(X\) is connected and Hausdorff.
Since \(G\) is a second-countable smooth manifold, it has the
homotopy type of a countable CW complex. Consequently
\(\pi_1(G,z_0)\) is countable. The fiber of the connected covering
corresponding to \(H\) is naturally identified, according to the
chosen convention, with a coset space of \(H\) in
\(\pi_1(G,z_0)\). Hence every fiber, and therefore the set of
sheets, is countable. 
Choose a countable basis \((U_m)_{m\geq1}\) of connected, evenly
covered open subsets of \(G\). For each \(m\), write
\[  p^{-1}(U_m)=\bigsqcup_{k\in I_m}V_{m,k},\]
where \(p|_{V_{m,k}}:V_{m,k}\to U_m\) is a homeomorphism.
The index set \(I_m\) is in bijection with the fiber over any
point of \(U_m\), so it is countable. Therefore $    \{V_{m,k}:m\geq1,\ k\in I_m\}$
is a countable basis for \(X\), and \(X\) is second-countable.
With the unique complex structure 
induced by the local biholomorphism  \(p\), $X$ is a complex manifold and thus $(X, p)$ a Riemann
domain over $\CC^n$.

Let
\[ Z_0=\{[w_0:w_1:\cdots]\in\PP^\infty:w_0=0\} \]
and define $  Z:=\wt{\cB}^{-1}(Z_0).$
Obviously, $Z$  is a proper analytic subvariety of \(X\) and it follows from the definition of $\cB$ that $Z \cap j(\Omega) =\emptyset$.  Moreover, $   X^\circ:=X\setminus Z$
is connected.

Let \(h\in\A(\Omega)\).  There is a unique coefficient sequence
\((c_j)\in\ell^2\) such that $  h=\sum_{j=0}^{\infty}c_j\phi_j  \in \A(\Omega)$.
Let \(V\subset X\) be an open subset on which
\(\wt{\cB}\) has a holomorphic \(\ell^2\)-valued lift
$$  F=(F_0,F_1,\ldots):V\rightarrow\ell^2\setminus\{0\}.$$
Namely, \(\wt{\cB}=[F_0,F_1,\ldots]: V \to \PP(\ell^2) =\PP^\infty\).
On \(V\cap X^\circ\), 
\begin{equation}\label{eq:h-extension}
  \wt h  :=\phi_0\frac{\sum_{j=0}^{\infty}c_jF_j}{F_0}.
\end{equation}
is a well-defined holomorphic function. Moreover, it is easy to verify that $\wt h$ is  independent of the lift and glues to a
global function $  \wt h\in\cO(X^\circ).$
Moreover,  \(\wt h\circ j=h\).  We have therefore shown
that every function in \(\A(\Omega)\) extends holomorphically to the Riemann domain \(X^\circ\).

Let $  \bigl(\wh{\Omega}_2,\wh p,\iota,  (\wh h)_{h\in\A(\Omega)}\bigr)$ be the \(L_h^2(\Omega)\)-envelope of holomorphy of $\Omega$ (cf. \cite{I04}), where  \(\wh p:\wh{\Omega}_2\to\C^n\) is a Riemann domain,
\(\iota:\Omega\hookrightarrow\wh{\Omega}_2\) is the distinguished
sheet, \(\wh p\circ\iota=\id_\Omega\), and
\(\wh h\circ\iota=h\) for every \(h\in\A(\Omega)\).
It follows from the argument above that $  \bigl(X^\circ,p|_{X^\circ},j,(\wt h)_{h\in\A(\Omega)}\bigr)$
is an  \(\A(\Omega)\)-extension of $\Omega$. 
Therefore there is a unique morphism $  \Phi:X^\circ\rightarrow\wh{\Omega}_2$
satisfying
\[  \Phi\circ j=\iota,  \quad \wh p\circ\Phi=p|_{X^\circ}.\]
It follows that
\(\wt h=\wh h\circ\Phi\) for every \(h\in\A(\Omega)\).
Moreover,
\[d\wh p_{\Phi(x)}\circ d\Phi_x=d p_x.\]
Since both $\wh p$ and $p$ are locally biholomorphic, 
\(d\Phi_x\) is an isomorphism for every \(x\in X^\circ\).  Thus \(\Phi\)
is locally biholomorphic.

Now let $  (\wh\Omega_{\cO},\wh p_{\cO},\iota_{\cO})$
be the ordinary envelope of holomorphy. There is the unique canonical 
morphism
\begin{equation}\label{eq:canonical-envelope-map}
  \kappa:\wh\Omega_{\cO}\rightarrow\wh\Omega_2,
  \quad
  \wh p\circ\kappa=\wh p_{\cO},
  \quad
  \kappa\circ\iota_{\cO}=\iota.
\end{equation}
Since $\Omega$ is pseudoconvex, the following
identification holds: 
\(\wh\Omega_{\cO}=\Omega\),
\(\iota_{\cO}=\id_\Omega\). Also
\eqref{eq:canonical-envelope-map} implies
\(\kappa(\wh\Omega_{\cO})=\iota(\Omega)\).  Therefore it follows from Theorem 9 in \cite{I04}
 that
\[  P:=\wh{\Omega}_2\setminus\iota(\Omega)\]
is pluripolar in \(\wh{\Omega}_2\).  If
\(x\in X^\circ\) and \(p(x)\in G\setminus\Omega\), then
\(\Phi(x)\notin\iota(\Omega)\), because
\(\wh p(\Phi(x))=p(x)\notin\Omega\).  Hence
\begin{equation}\label{eq:exceptional-containment}
  p^{-1}(G\setminus\Omega)\cap X^\circ
  \subset \Phi^{-1}(P).
\end{equation}
Because \(\Phi\) is locally biholomorphic, the inverse image
\(\Phi^{-1}(P)\) is locally pluripolar.

Fix \(q\in G\setminus\Omega\), and choose \(x\in p^{-1}(q)\).
Take an open subset \(V\ni x\) in $X$ such that $  p|_V:V\rightarrow U$
is biholomorphic onto an open set \(U\ni q\).  Let $  A:=\Phi^{-1}(P)\subset X^\circ.$ 
The distinguished sheet \(\iota(\Omega)\) is open, so \(P\) is closed
and hence \(A\) is closed in \(X^\circ\). 
We will show that  \(A\cap V\)  is pluripolar in \(V\), by 
identifying \(V\) with the domain \(U\subset\C^n\) by \(p|_V\) and applying
Josefson's theorem. 
Choose a countable cover
\((W_m)_{m\geq1}\) of \(V\setminus Z\) by open subsets relatively
compact in \(V\setminus Z\). 
Because \(A\) is closed in \(X^\circ\) and
\(\overline{W_m}\Subset V\setminus Z\), every accumulation point of
\(A\cap W_m\) in \(V\) lies in \(A\cap\overline{W_m}\).
Josefson's theorem (cf. \cite{J78}) implies that 
each \(A\cap W_m\) is locally pluripolar in \(\C^n\), hence pluripolar in
\(V\). Since a countable union of pluripolar sets is
pluripolar, $ A\cap V=\bigcup_{m\geq1}(A\cap W_m)$ is pluripolar in
\(V\).  Since \(Z\cap V\) is analytic, it is
pluripolar as well, and so is their union. By \eqref{eq:exceptional-containment},
$ p^{-1}(G\setminus\Omega)\cap V \subset (A\cap V)\cup(Z\cap V),$ and thus $  (G\setminus\Omega)\cap U$
is pluripolar.  Since \(q\) was arbitrary, \(G\setminus\Omega\) is
locally pluripolar in \(\C^n\).
Josefson's theorem \cite{J78} implies that a locally pluripolar
subset of \(\C^n\) is globally pluripolar.  Therefore
\(G\setminus\Omega\) is pluripolar.
\end{proof}

The following lemma is standard and known to experts (cf. \cite{DW22a, HL25a}).

\begin{lemma}\label{lem:kernel-removal}
Let \(G\subset\C^n\) be a bounded domain, and let \(E\subset G\) be relatively
closed and pluripolar.  Assume \(G\setminus E\) is connected.  Then
restriction is a unitary isomorphism from
\(\A(G)\) to \( \A(G\setminus E)\)
and \(  K_G(z,w)=K_{G\setminus E}(z,w)\) for all \( z,w\in G\setminus E.\)
Consequently,
\( \omega_G|_{G\setminus E}=\omega_{G\setminus E}. \)
\end{lemma}

\begin{proof}
A pluripolar subset of \(\C^n\) has Lebesgue measure zero.  Let
\(h\in\A(G\setminus E)\).  
The \(L^2\)-removability
theorem for relatively closed pluripolar
sets (cf. \cite[Lemma~1]{I04}, \cite{S82}) yields a
unique holomorphic extension \(H\in\cO(G)\).  Because \(E\) has
measure zero,
\(  \|H\|_{L^2(G)}=\|h\|_{L^2(G\setminus E)}. \)
Thus \(H\in\A(G)\), and restriction is a surjective isometry.  It is
injective by the identity theorem, so it is unitary. It follows from the standard argument that the Bergman kernel functions and the Bergman metrics are equal.
\end{proof}

Theorem \ref{main1} now follows from combining Calabi's extension theorem (cf. \cite{HL25b} by Huang-Li) and Proposition \ref{lem:continuation}.

\begin{proof}[Proof of Theorem~\ref{main1}]
Let \( \Omega=f(M) \subset D.\)
It follows from Proposition 3.1 in \cite{HL26} that the Bergman space $\A(M)$ separates points of $M$. 
By Proposition \ref{inj}, \(f: M \to\Omega\) is biholomorphic. By the assumption of Theorem \ref{main1},
\(\Omega\) is bounded pseudoconvex  and 
 $ \omega_{D}|_\Omega=\lambda\omega_\Omega.$
Define $  \eta:=\lambda^{-1}\omega_{D}$ and then 
\( \eta|_\Omega=\omega_\Omega. \)
By \eqref{eq:BB-pullback}, the Bergman-Bochner map of \(\Omega\) is
a holomorphic isometric map from
\((\Omega,\eta|_\Omega)\) into
\((\PP^\infty,\omega_{\FS})\).  Choose \(z_0\in\Omega\) and a
neighborhood \(U_0\Subset\Omega\) of \(z_0\).  Applying the Calabi
extension theorem (cf. \cite[Theorem~1.1]{HL25b}) to the restriction
of this map to \(U_0\) on the real-analytic K\"ahler manifold
\((D,\eta)\), the germ at \(z_0\) of this projective isometry
continues holomorphically and isometrically along every continuous
curve in \(D\) starting at \(z_0\).
Proposition~\ref{lem:continuation} 
 implies that \( E:=D\setminus\Omega \)
is pluripolar.  Moreover, $E$ is relatively closed in \(D\), since \(\Omega\)
is open.
It follows from Lemma~\ref{lem:kernel-removal} that
  $\omega_{D}|_\Omega=\omega_\Omega.$
Therefore \( \lambda=1 \) and 
\(f: M_1\to D \setminus E\) is biholomorphic.
\end{proof}

\begin{proposition}\label{prop:stein-target-continuation}
Let $D\subset\C^n$ be a bounded pseudoconvex domain, let $M$ be a
complex manifold of complex dimension $n$. Assume 
$    f:D\rightarrow\Omega\subset M$
be a biholomorphism from $D$ onto an open subset $\Omega$ in $M$. Suppose that
 $    \mathcal B_D\circ f^{-1}:\Omega\rightarrow\PP^\infty$
admits holomorphic continuation along every continuous curve in $M$.
Then $M\setminus\Omega$ is locally pluripolar.
\end{proposition}

\begin{proof}
We follow the proof of Proposition~\ref{lem:continuation}. The only
additional point is that the continuation space is initially a
Riemann domain over $M$, rather than over $\C^n$. 
Indeed, retaining all branches of the continuation gives a connected
covering Riemann domain
$  \pi:X\rightarrow M$
and a single-valued holomorphic map $    F:X\rightarrow\PP^\infty.$
Put $g=f^{-1}$. 
There is a canonical lift $\iota:\Omega\to X$ satisfying
\[
   \pi\circ\iota=\id_\Omega,
    \quad
    F\circ\iota=\mathcal B_D\circ g.
\]
Choose an orthonormal basis $(\phi_j)_{j\geq0}$ of $A^2(D)$ such that
$ \phi_0\equiv c:=\Vol(D)^{-1/2}.$
Let 
$$    Z_0:=F^{-1}\{[w_0:w_1:\cdots]\in\PP^\infty:w_0=0\}.$$
Then $Z_0$ is a proper analytic subset of $X$ and
$Z_0 \cap\iota(\Omega)=\varnothing$. On $X^0:=X\setminus Z_0$, write
$  F=[1:h_1:h_2:\cdots].$
Exactly as in the proof of Proposition~\ref{lem:continuation}, if $    u=\sum_{j\geq0}a_j\phi_j\in A^2(D),$
then
\[ \widehat u := c\left(a_0+\sum_{j\geq1}a_jh_j\right)\]
defines a holomorphic function on $X^0$, and satisfies $    \widehat u\circ\iota=u\circ g.$
Thus every function in $A^2(D)$ extends as a single-valued holomorphic function on $X^0$.
Here is the additional observation needed when the target is a complex manifold.
Since $D$ is bounded, its coordinate functions
$z_1,\ldots,z_n$ belong to $A^2(D)$. Their continuations therefore
define a holomorphic map
\[   G:=(\widehat z_1,\ldots,\widehat z_n):
    X^0\rightarrow\C^n,
    \quad
    G\circ\iota=g.\]
Let
$$C:=\{x\in X^0:\det dG_x=0\}.$$
$C$ is a proper analytic subset, because $G=g$ on the canonical sheet.
Consequently,
$  X^1:=X^0\setminus C$
is connected and $(X^1,G)$ is a Riemann domain over $\C^n$. Moreover,
the map
\[   j_0:=\iota\circ f:D\rightarrow X^1\]
satisfies
$ G\circ j_0=\id_D,$ 
and every element of $A^2(D)$ extends holomorphically to $X^1$.

Let
\[   j:D\rightarrow\widehat D_2,  \quad   \widehat p_2:\widehat D_2\rightarrow\C^n\]
be the $A^2(D)$-envelope of holomorphy of $D$. The universal property yields a holomorphic
map $    \Phi:X^1\rightarrow\widehat D_2$
such that
\[
    \Phi\circ j_0=j,
    \quad
    \widehat p_2\circ\Phi=G.
\]
Notice that $\Phi$ is locally biholomorphic, since
\[   d\widehat p_{2,\Phi(x)}\circ d\Phi_x=dG_x\]
and both $G$ and $\widehat p_2$ are locally biholomorphic.

Because $D$ is pseudoconvex, its ordinary envelope of holomorphy is
$D$ itself. Irgens's theorem therefore implies that
\[  S:=\widehat D_2\setminus j(D)\]
is pluripolar. Hence $ R:=\Phi^{-1}(S)$
is a relatively closed pluripolar subset of $X^1$. Moreover,
$X^1\setminus R$ is connected and, for every $x\in X^1\setminus R$, $    G(x)=\widehat p_2(\Phi(x))\in D.$
The two holomorphic maps
\[  \pi,\ f\circ G:X^1\setminus R\rightarrow M \]
agree on the nonempty open subset $j_0(D)$. The identity theorem gives $    \pi=f\circ G$ on  $X^1\setminus R.$
In particular, $    \pi(X^1\setminus R)\subset f(D)=\Omega.$
Set $    Q:=Z\cup C\cup R\subset X.$
Then $ M\setminus\Omega\subset\pi(Q).$
We briefly justify the pluripolarity assertion across the successive
exceptional sets. Fix a coordinate ball $V\subset X$. Cover
$V\setminus Z$ by countably many relatively compact open sets $W_\nu$
and cover $V\setminus(Z\cup C)$ by countably many relatively compact
open sets $U_\mu$. The sets $    C\cap W_\nu, R\cap U_\mu$
are locally pluripolar in $V$, while $Z\cap V$ is analytic. Josefson's
theorem in the coordinate ball $V$ shows that each of these sets is
locally pluripolar in $V$. Hence $Q\cap V$ is locally pluripolar, and therefore $Q$ is
locally pluripolar in $X$.
Finally, cover $X$ by countably many open sets on which $\pi$ is
biholomorphic. It follows that $\pi(Q)$ is locally pluripolar in $M$.
Consequently, $    M\setminus\Omega\subset\pi(Q)$
is locally pluripolar.
\end{proof}

Now, Theorem \ref{thm:main2} follows verbatim the same argument as Theorem \ref{main1}, using the following lemma, whose proof is also similar to that of Lemma \ref{lem:kernel-removal}.

\begin{lemma}\label{lem:manifold-removal}
Let $Y$ be an $n$-dimensional complex manifold and let $P\subset Y$
be relatively closed and locally pluripolar. Then restriction induces a
unitary isomorphism $    A^2_{(n,0)}(Y)\rightarrow A^2_{(n,0)}(Y\setminus P).$
Consequently, $ \omega_Y|_{Y\setminus P}=\omega_{Y\setminus P}.$
\end{lemma}

\begin{remark}
Theorem~\ref{thm:main2} shows that $E=M\setminus f(D)$ is locally
pluripolar. If $M$ is Stein, then the Stein version of Josefson's
theorem  implies that $E$ is globally pluripolar in $M$.
\end{remark}

\begin{remark}
Let $ M:=(\Delta\setminus\{0\}) \times\Delta$ be a bounded pseudoconvex domain, $D:=\Delta^2\setminus\{(0,0)\}$ not pseudoconvex.
Then $f: M\rightarrow D$ given by $f(z_1,z_2)=(z_1,z_2)$ satisfies $f^*\omega_{D}=\omega_{M}$ by Lemma \ref{lem:kernel-removal}.
This shows that $D$ is not necessarily Stein.
\end{remark}

\begin{corollary}\label{al1}
Let $D_1, D_2 \subset \mathbb{C}^n$ be bounded domains.
Assume that $(D_1, \omega_{D_1})$ is complete.  
Then the holomorphic map $f: D_1 \rightarrow D_2$ satisfying $f^*\omega_{D_2} = \lambda \omega_{D_1}$ for some constant $\lambda>0$ is a biholomorphism. In particular, $\lambda=1$.
\end{corollary}

\begin{proof}
Since $(D_1, \omega_{D_1})$ is complete, $D_1$ is pseudoconvex \cite{Br55}. By Theorem \ref{main1}, there exists a relatively closed pluripolar set $E$ in \(D_2\), 
such that
$  f:D_1\rightarrow D_2\setminus E $
is biholomorphic. By Lemma \ref{lem:kernel-removal}, $\omega_{D_2}=\omega_{D_2\setminus E}$. 
If \((D_1, \omega_{D_1})\) is complete, then
$   \bigl(D_2\setminus E, \omega_{D_2}|_{D_2\setminus E}\bigr) $
is complete because it is isometric to \((D_1, \omega_{D_1})\). 
If \(E\neq
\varnothing\), any sequence of points $\{ z_j\} \subset D_2\setminus E$ with $z_j \to z_0 \in E$ is a Cauchy sequence in $\omega_{D_2}$ and converges to $z_0$. 
This contradicts
completeness. 
Thus \(E=\varnothing\).

We can also prove the result without using Theorem \ref{main1}. By Proposition \ref{inj}, it suffices to
  show that $f$ is surjective. Suppose not. Namely, there exists $ D_2\setminus f(D_1) \not= \emptyset$. Since ${\rm Jac}_{\mathbb{C}}(f)$ is non-degenerate everywhere, $f(D_1)$ is a connected, open set. Then there exists $w \in  D_2\setminus f(D_1)$ and a continuous curve $\gamma: [0, 1) \rightarrow f(D_1)$ such that $\gamma(t)  \rightarrow w$ as $t \rightarrow 1$. Let $t_j=1-\frac{1}{j}$. As $\omega_{D_2}$ is quasi-isometry to the Euclidean metric near $w \in D_2$, $\{\gamma(t_j)\}_{j=1}^\infty$ is a Cauchy sequence in $(D_2, \omega_{D_2})$. Let $f(z_j)=\gamma(t_j)$. It follows from the isometry condition that $\{z_j\}_{j=1}^\infty$ is a Cauchy sequence in $(D_1, \omega_{D_1})$. Since $(D_1, \omega_{D_1})$ is complete, there exists $z \in D_1$, such that $z_j \rightarrow z$ as $j \rightarrow \infty$. By continuity, $f(z)=\lim_{j \rightarrow \infty} f(z_j)=\lim_{j\rightarrow \infty} \gamma(t_j)=w$. This contradiction shows that $f$ is surjective.
\end{proof}

\begin{remark}
It is well-known that  if the bounded domain  $(D, \omega_{D})$ is complete, then $D$ is a $L^2$-domain of holomorphy. However, 
the assumption of the completeness of $(D_1, \omega_{D_1})$  cannot be weakened to the $L^2$-domain of holomorphy. For example, the Hartogs triangle $\mathbb{H}:=\{(z_1, z_2) \in \mathbb{C}^2: |z_1|<|z_2|<1\}$ is an $L^2$-domain of holomorphy while the Bergman metric is not complete. Note that $\Phi: \mathbb{H} \to \Delta \times \left( \Delta \setminus \{0\} \right)$ given by $\Phi(z_1, z_2)=(\frac{z_1}{z_2}, z_2)$ is a biholomorphism. Then $\Phi: \mathbb{H} \rightarrow \Delta^2$ satisfies $\Phi^*ds^2_{\Delta^2} = ds^2_{\mathbb{H}}$ as $ds^2_{\Delta^2}=ds^2_{\Delta\times \left( \Delta \setminus \{0\} \right)}$. 
\end{remark}

\subsection{Local holomorphic Bergman isometries}

Assuming that the Bergman metric is complete, we obtain the 
 local version of Theorem \ref{main1} and Theorem \ref{thm:main2} using the deep extension results due to Mok \cite{M12} and Huang-Li \cite{HL26}.

 \begin{proof}[Proof of Theorem \ref{al2}]
 The proof follows by combining Corollary \ref{yy} and Theorem \ref{main1}.
 \end{proof}

 \begin{proof}[Proof of Theorem \ref{al3}]
 The proof follows by combining Corollary \ref{yy2} and Theorem \ref{thm:main2}.
 \end{proof}

 \begin{corollary}
 Let $D_1, D_2 \subset \mathbb{C}^n$ be bounded domains. 
Assume that $(D_1, \omega_{D_1})$ and $(D_2, \omega_{D_2})$ are complete. Then
for any connected open set $U \subset D_1$, the holomorphic map $f: U \rightarrow D_2$ satisfying $f^*\omega_{D_2} = \lambda \omega_{D_1}$ on $U$ for some constant $\lambda>0$ must extend to a biholomorphism.
\end{corollary}

\begin{proof}
If $(D_2, \omega_{D_2})$ is complete, $f$ extends to a holomorphic Bergman isometry from $D_1$ to $D_2$ by Mok's extension theorem \cite{M12}.
Then the conclusion follows from Corollary \ref{al1}. The alternative way is to apply Mok's extension theorem \cite{M12} to the local inverse of $f$, yielding that the local inverse of $f$ extends to holomorphic Bergman isometry from $D_2$ to $D_1$. Moreover, $f$ is a biholomorphism by the identity theorem for holomorphic functions.
\end{proof}

\begin{remark}
Corollary \ref{al1} apply to the case when $D_1, D_2$ are hyperconvex \cite{BP98, H99}. This includes many examples, such as pseudoconvex domains with Lipschitz boundary.
\end{remark}

\subsection{General case}

Inspired by the proof of Proposition 1.4 in \cite{ETX25b} and the proof of Proposition 4.1 in \cite{BGNX26}, we obtain a partial generalization of Theorem \ref{domain}. We first recall the following definition introduced by Ebenfelt-Treuer-Xiao in \cite{ETX25b}.

\begin{definition}[Definition 1.3 in \cite{ETX25b}]
Let $U$ be a domain in $\mathbb{C}^n$. A relatively closed (possibly empty) subset $F \subset U$ is called a Bergman-negligible subset
of $U$ if $F$ is of Lebesgue measure zero, and every $L^2$-holomorphic function on $U \setminus F$ extends as an $L^2$-holomorphic function on $U$.
\end{definition}

\begin{proposition}\label{negligible}
Let \(D_1,D_2\subset\C^n\) be bounded domains, and let
\( f:D_1\rightarrow D_2 \)
be a bona fide
 holomorphic Bergman isometry satisfying
\[  f^*\omega_{D_2}=\omega_{D_1}. \]
Then 
\( E:=D_2\setminus f(D_1)\)
is Bergman-negligible  in \(D_2\). In particular,
\( f:D_1\rightarrow D_2\setminus E \)
is a biholomorphism.
\end{proposition}

\begin{proof}
By Proposition \ref{inj}, \(f:D_1\to\Omega:=f(D_1)\) is biholomorphic. 
Fix $\xi \in \Omega$. 
By the similar argument as in the proof of Proposition \ref{constant},
it follows from 
 \eqref{eq:orders}, applied to
 $\lambda=1$,  that both \(\frac{K_{D_2}(z, \xi)}{K_{\Omega}(z, \xi)}\) and \(\frac{K_{\Omega}(z, \xi)}{K_{D_2}(z, \xi)}\) have no pole
divisor and therefore extend holomorphically across their common
zeros to $\Omega$. The extensions are reciprocal to each other. Hence
$$  h_\xi(z):=
  \left(\frac{K_{\Omega}(\xi, \xi)}{K_{D_2}(\xi, \xi)}\right)^{\frac{1}{2}}
  \frac{K_{D_2}(z, \xi)}{K_{\Omega}(z, \xi)} $$
is a nowhere-vanishing holomorphic function on \(\Omega\).
It follows from \eqref{eq:normalized} with \(w= \xi\) that
 $$ K_{D_2}(z, z)=|h_\xi(z)|^2 K_{\Omega}(z,z)$$
holds for every $z \in \Omega$.
By polarization, we have
\begin{equation}\label{eq:multiplier}
  K_{D_2}(z, w)=h_\xi(z) \overline{h_\xi(w)} K_{\Omega}(z, w)
\end{equation}
for all $z, w \in\Omega. $

By the identity theorem for holomorphic functions,  
the restriction from \(\A(D_2)\) to the non-empty open set \(\Omega\)
is injective.  Define
\[  \mathcal R:= \{F|_{\Omega}:F\in\A(D_2)\},  \quad  \|F|_{\Omega}\|_{\mathcal R} := \|F\|_{\A(D_2)}.\]
Then \(\mathcal R\) is a reproducing-kernel Hilbert space on
\(\Omega\), and its reproducing kernel is
\( K_{D_2}|_{\Omega\times\Omega}. \)
Let \(  M_{h}: \A(\Omega)\rightarrow\mathcal R \) be given by 
\( M_{h}(g):= h_\xi  g \).
Note that the linear spans of the functions \(K_\Omega(z, w_j)\) and \(K_{D_2}(z, w_j)\), for $w_j \in \Omega$, are
dense in \(\A(\Omega)\) and \(\mathcal R\), respectively. By
\eqref{eq:multiplier}, we have 
$  M_h \left(K_\Omega(z, w_j) \right)= \frac{K_{D_2}(z, w_j)} {\overline{h_\xi(w_j)}}.$
For \(w_j, w_k \in\Omega\), the reproducing property and
\eqref{eq:multiplier} yield
\[
\begin{aligned}
~&  \left\langle M_h(K_\Omega(z, w_j)),
                   M_h(K_\Omega(z, w_k))\right\rangle_{\mathcal R} \\
  &=
  \frac{\langle K_{D_2}(z, w_j), K_{D_2}(z, w_k)\rangle_{\mathcal R}}
       {\overline{h_\xi(w_j)}h_\xi(w_k)} =
  \frac{K_{D_2}(w_k, w_j)}
       {\overline{h_\xi(w_j)}h_\xi(w_k)} =
  K_\Omega(w_k, w_j) =
  \left\langle K_\Omega(z, w_j), K_\Omega(z, w_k)
  \right\rangle_{\A(\Omega)}.
\end{aligned}
\]
Consequently, \(M_h\) is  an isometry with a closed range.
Moreover,
since the linear span
of \(K_{D_2}(z, w_j) = \overline{h_\xi(w_j)}\,M_h (K_\Omega(z, w_j))\) is dense in \(\mathcal R\), the range of
\(M_h\) is also dense. Hence \(M_h\) is surjective and therefore
unitary.

Since \(1\in\A(\Omega)\),  \(h_\xi=M_h(1) \in \mathcal R\). Hence there exists a unique
\(\widetilde H\in\A(D_2)\) such that $  \widetilde H|_{\Omega}=h.$
Let \(P\) be a holomorphic polynomial. Since \(D_2\) is bounded,
\(\widetilde H P\in\A(D_2)\), and it is the unique extension of \(h_\xi P\) to
\(D_2\). Moreover, 
\[  \int_{D_2}|\widetilde H|^2|P|^2\,dV =  \int_{\Omega}|P|^2\,dV.\]
By polarization, for all holomorphic
polynomials \(P,Q\),
\begin{equation}\label{eq:moments}
  \int_{D_2}   |\widetilde H|^2P\overline{Q}\,dV =  \int_{\Omega}    P\overline{Q}\,dV
\end{equation}
holds.
Define $d\mu :=  |\widetilde H|^2\mathbf 1_{D_2}\,dV$ and $d\nu  :=  \mathbf 1_{\Omega}\,dV$ to be
 the finite Borel measures on \(\overline{D_2}\). The complex $*$-algebra generated by 1, the holomorphic coordinate functions 
\(z_1,\ldots,z_n\), and their conjugates is dense in
\(C(\overline{D_2})\) by the complex version of the Stone-Weierstrass theorem.
Every monomial in \(z\) and \(\overline z\) is of the form
\(P\overline Q\). Hence \eqref{eq:moments} implies that
\(\mu=\nu\). Namely, 
\begin{equation}\label{eq:measures}
  |\widetilde H|^2\mathbf 1_{D_2}\,dV =  \mathbf 1_{\Omega}\,dV.
\end{equation}
It follows that $  |\widetilde H|=1$ 
almost everywhere on \(\Omega\). Since \(\widetilde H\) is a holomorphic function on $D_2$, 
there exists
\(c\in\C\), with \(|c|=1\), such that
\( \widetilde H\equiv c\)
on $D_2.$
It follows from \eqref{eq:measures} that $  \mathbf 1_{D_2}\,dV=  \mathbf 1_{\Omega}\,dV,$
and therefore $  \Vol(D_2\setminus\Omega)=0.$
Furthermore, \(h=\widetilde H|_{\Omega}=c\) and thus  $ \mathcal R=\A(\Omega)$
with equality of norms. 
In particular, every function in \(\A(\Omega)\) extends
uniquely to a function in \(\A(D_2)\).
By definition,  $ E:=D_2\setminus\Omega$
 is Bergman-negligible. 
\end{proof}

\subsection{Applications}

We now state some consequences. For simplicity of notation, we consider only bounded domains in \(\mathbb{C}^n\) throughout this section. The arguments extend naturally to Stein manifolds by the preceding argument.
Let \(D_1,D_2\subset\mathbb C^n\), \(n\geq 2\), be bounded
 domains. For \(j=1,2\), let \(\omega_j=\omega_{D_j}\) be the
Bergman metric, 
and denote by $\nabla_j, R_j, {\rm Ric}_j$
the Levi-Civita connection, the curvature tensor, and the Ricci form
of \(\omega_j\), respectively, where the curvature tensor \(R_j\) is
regarded as a tensor of type \((1,3)\).

\begin{definition}
A local biholomorphic map $f: D_1 \to D_2$ is called strongly curvature-preserving if $f$ preserves the full curvature tensor and all of its covariant derivatives in the sense that 
$$f^*(\nabla_{2}^k R_2) = \nabla_{1}^k R_1$$
for $k=0, 1, 2, \cdots$.
\end{definition}

\begin{corollary}\label{curv1}
Let $D_1, D_2$ be bounded domains in $\CC^n$. Assume that $(D_1, \omega_1)$ and   $(D_2, \omega_2)$ are locally irreducible as Riemannian metrics.
 Assume $D_1$ to be pseudoconvex. If a holomorphic map $f: D_1 \to D_2$ is strongly curvature-preserving and a local biholomorphism from $D_1$ to $f(D_1)$, 
then there exists a relatively closed pluripolar set $E$ in \(D_2\), 
such that $  f:D_1\rightarrow D_2\setminus E $ is biholomorphic.
\end{corollary}

\begin{proof}
 Consider the metric $\omega'_1=f^*\omega_2$ and denote the Levi-Civita connection, the curvature tensor of $\omega'_1$ by $\nabla'_1, R'_1$. 
Obviously the identity map $i: D_1 \to D_1$ is strongly curvature-preserving. 
By the classical result of Nomizu-Yano (Theorem A in \cite{NY67}), $i^*\omega'_1=\lambda \omega_1$ for $\lambda>0$, thus $f^*\omega_2=\lambda\omega_1$.  It thus follows from Theorem \ref{domain} that there exists a relatively closed pluripolar set $E$ in \(D_2\), 
such that $  f:D_1\rightarrow D_2\setminus E $ is biholomorphic.
\end{proof}

\begin{corollary}\label{curv}
Let $D_1, D_2$ be bounded domains in $\CC^n$. Assume that $(D_1, \omega_1)$ and   $(D_2, \omega_2)$ are locally irreducible as Riemannian metrics.
Assume $(D_j, \omega_j)$ to be complete for $j=1, 2$. For an open connected subset $U \subset D_1$, if a local biholomorphic map $f: U \to D_2$ is strongly curvature-preserving, 
then $f$ extends to a biholomorphism from $D_1$ to $D_2$.
\end{corollary}

\begin{proof}
Note that the local analytic-holonomy argument underlying Nomizu-Yano in \cite{NY67}
yields $f^*\omega_2=\lambda\omega_1$. By Proposition \ref{al2}, $f$ extends to a biholomorphism from $D_1$ to $D_2$.
\end{proof}

\begin{corollary}\label{ric1}
Let $D_1, D_2$ be bounded domains in $\CC^n$. 
 Assume that $(D_1, \omega_1)$ and   $(D_2, \omega_2)$ are both K\"ahler-Einstein manifolds (not necessarily complete) with non-zero scalar curvature and $D_1$ is pseudoconvex. If a holomorphic map  $f: D_1 \to D_2$ satisfies $f^* {\rm Ric}_2={\rm Ric}_1$, then there exists a relatively closed pluripolar set
$   E\subset D_2 $ such that $   f:D_1\rightarrow D_2\setminus E $ is biholomorphic and
$  f^*\omega_2=\omega_1.$
\end{corollary}

\begin{proof}
 Suppose ${\rm Ric}_1=\kappa_1\omega_1,         {\rm Ric}_2=\kappa_2\omega_2,$ with $    \kappa_1\kappa_2\neq 0.$
Then $   f^*\omega_2      = \frac{\kappa_1}{\kappa_2}\,\omega_1.$ In particular $ \frac{\kappa_1}{\kappa_2} >0$. By Theorem \ref{domain}, there exists a relatively closed pluripolar set
$   E\subset D_2 $ such that $   f:D_1\rightarrow D_2\setminus E $ is biholomorphic and 
$ \frac{\kappa_1}{\kappa_2} =1$.
\end{proof}

\begin{corollary}\label{ric2}
Let $D_1, D_2$ be bounded domains in $\CC^n$.  
Assume that $(D_1, \omega_1)$ and   $(D_2, \omega_2)$ are both complete K\"ahler-Einstein manifolds. For an open connected subset $U \subset D_1$, if a local holomorphic map 
 $f: U \to D_2$  satisfies $f^* {\rm Ric}_2={\rm Ric}_1$ on $U$,  
then $f$ extends to a biholomorphism from $D_1$ to $D_2$ and $  f^*\omega_2=\omega_1.$
\end{corollary}

\begin{proof}
 It follows from \cite{Br55, CY80, MY80} that a bounded domain in $\CC^n$ admitting a complete K\"ahler-Einstein metric must be pseudoconvex and the scalar curvature is negative. 
Then the conclusion follows from combining the proof of the first part and the proof of Corollary \ref{curv}.
\end{proof}

\subsection{Further discussions}
In this section, we first show that for any $n \in \mathbb{N}$, there exist a bounded pseudoconvex domain $D$, and 
a holomorphic Bergman isometry $f: D\to D$ with isometric constant $\lambda=1$ that is not  a holomorphic automorphism of $D$.
  In this sense, Theorem \ref{domain} is optimal even for a self map.

We first construct an example in one complex dimension.
Let $ \Delta$ be the unit disc. Fixing a number \(a\in(0,1)\), let
$  A(\zeta)=\frac{\zeta+a}{1+a\zeta} \in\operatorname{Aut}(\Delta)$ and 
 $     p_j:=A^{\circ j}(0)$, for $j\in\mathbb Z$.
Equivalently, if \(\alpha=\operatorname{arctanh} a\), then 
$        p_j=\tanh(j\alpha).$ 
In particular, $        p_j\to -1$ as $ j\to -\infty$ and 
    $    p_j\to 1 $ as $ j\to +\infty.$
Define the discrete subset
$  E:=\{p_j:j\le 0\}\subset\Delta,$
and set $        D_0:=\Delta\setminus E.$
Since \(E\) has no accumulation point in \(\Delta\), \(D_0\) is a bounded domain.
It is easy to see that the Bergman metric of \(D_0\) is just the restriction of the
Bergman metric of \(\Delta\). Indeed, every function
$  h\in A^2(D_0)$
extends holomorphically across the isolated punctures \(E\) to a holomorphic function \(\widetilde h\) on $\Delta$.Moreover, \(E\) has Lebesgue measure zero, so
$      \|h\|_{A^2(D_0)}=\|\widetilde h\|_{A^2(\Delta)}.$
Therefore
$   A^2(D_0)=A^2(\Delta)$
isometrically, after restriction to \(D_0\). Consequently,
$  K_{D_0}(z,w)=K_\Delta(z,w)\big|_{D_0\times D_0},$
and hence
$   ds^2_{D_0}=ds^2_\Delta\big|_{D_0}.$
Now consider
\[        f:=A|_{D_0}:D_0\to  D_0.\]
Since $        A^{-1}(E)=        \{p_j:j\le -1\}   \subset E$, $f$ is well-defined.
Thus, if \(z\notin E\), then \(A(z)\notin E\). Hence \(A(D_0)\subset D_0\).
Since \(A\) is an automorphism of \(\Delta\), it preserves the Bergman metric of
the disc $\Delta$. 
Therefore \(f\) is a holomorphic Bergman isometry from $D_0$ to $D_0$ with isometry constant 1. Namely, 
\[        f^* ds^2_{D_0}    =   A^*\bigl(ds^2_\Delta|_{D_0}\bigr)   =  ds^2_\Delta|_{D_0}   =  ds^2_{D_0}.\]
Moreover, 
\[ A(E) =   \{A(p_j):j\le0\}   =  \{p_{j+1}:j\le0\} = \{p_j:j\le1\}\]
implies
\[ f(D_0)  =A(\Delta\setminus E) = \Delta\setminus A(E) = \Delta\setminus\{p_j:j\le1\} =D_0\setminus\{p_1\}.\]
In particular,  \(f\) is not onto.

Let  $ D:=D_0\times\Delta^{n-1}\subset\mathbb C^n,$
and define
$   F:D \to D $
by
$   F(z_1,z_2,\dots,z_n)  =  \bigl(f(z_1),z_2,\dots,z_n\bigr).$
It follows that  $  F(D)=D\setminus \bigr(\{p_1\} \times\Delta^{n-1}\bigr)$ and thus 
 $F$ is not onto.

\section{Rigidity for local holomorphic Bergman isometries and application to analytic correspondence}
\subsection{Rigidity for local holomorphic Bergman isometries}

\begin{proof}[Proof of Theorem~\ref{product}]
Write $F=(f_1, \cdots, f_m), \Omega=D^m$. 
By Mok's extension theorem (cf. Theorem 2.1.2 in \cite{M12}), $F$ extends to a proper holomorphic map, still denoted by $F=(f_1, \cdots, f_m)$, from $D$ to $\Omega$, since $\omega_D$ and $\omega_\Omega$ are complete. 
By Theorem III.15 of Zorn~\cite{Z82}, for every compact subset
\(L\Subset D\), there exists a domain \(D_L^\sharp\supset\overline D\)
such that the Bergman kernel \(K_D(z,\bar w)\) admits a joint extension
to a neighborhood of \(D_L^\sharp\times L\), holomorphic in \(z\) and
anti-holomorphic in \(w\). 
The same conclusion holds for $\Omega$ as well. 
The proof of
Theorem 2.1.2 by Mok in \cite{M12} uses only this compact-parametric, one-sided
extension property for \(w\) in a sufficiently small neighborhood of
the base point, together with the analogous property for the target.
Consequently, Mok's argument yields that 
there are open sets \(U \supset\overline D\) and
  \( V \supset\overline\Omega\), and an irreducible
  \(n\)-dimensional complex-analytic subvariety $     S \subset U \times V$,
  such that \(S\cap(D\times\Omega)=\Graph(F)\).

  Let \(\pi:S\to U\) be the first projection.  Define the bad locus
\[ B:=\Sing S\ \cup\ \bigl\{x\in\Reg S: {\rm rank}_{\C}d\pi_x<n\bigr\}.\]
On the interior graph \(\Graph(F)\), the restriction of \(\pi\) is
the inverse of \(z\mapsto(z,F(z))\), so it has rank \(n\).
Because \(S\) is irreducible, \(\Reg S\) is connected.  
Consequently \(B\) is  a complex-analytic subset
of complex dimension at most \(n-1\).
Since the analytic set \(S\), and thus 
\(B\), is second countable, we may cover \(B\) by
countably many relatively compact coordinate pieces of its analytic
strata.  On every such piece the coordinate projection \(\pi\) is
locally Lipschitz.  Therefore the real Hausdorff
dimension of $\pi(B)$ is at most $2n-2.$
Since the smooth hypersurface \(\partial D\) has real Hausdorff
dimension \(2n-1\), one may choose
$ p\in\partial D\setminus\pi(B).$
Take any sequence \(z_\nu\in D\) with \(z_\nu\to p\).
After passage to a subsequence, 
$ F(z_\nu)\rightarrow q\in \partial \Omega$
since $F$ is proper.
Properness excludes \(q\in\Omega\), so \(q\in\partial\Omega\).
Because \(S\) is relatively closed in \(U\times V\),
\((p,q)\in S\).  Furthermore, \(S\) is smooth at
\((p,q)\) and \(d\pi_{(p,q)}\) is invertible.  The holomorphic
inverse function theorem yields the corresponding local piece of
\(S\) as $ \bigl\{(z,\widehat F(z)):z\in U_p\bigr\}$
for a holomorphic map \(\widehat F\), with
\(\widehat F(p)=q\).
For all sufficiently large \(\nu\), the points
\((z_\nu, F(z_\nu))\) lie on this local sheet.  Near each such point
the sheet lies in \(D\times\Omega\), where \(S\) is exactly
\(\Graph(F)\).  Thus \(\widehat F=F\) on a nonempty open subset of
\(U_p\cap D\).  Shrinking \(U_p\) so that \(U_p\cap D\) is connected,
the identity theorem yields $  \widehat F=F$ on $U_p\cap D. $
Again since $F$ is proper,  $ \widehat F(U_p\cap\partial D)\subset\partial\Omega.$

By the Baire category theorem, there exist an open neighborhood $W$ of a point $p' \in \partial D$ and some $1 \leq j_0 \leq m$, such that $f_{j_0}(W \cap \partial D) \subset \partial D$. The standard Hopf lemma yields that $f_{j_0}$ is transversal at $p'$ and thus is a local biholomorphism. It follows from Theorem A in \cite{NS05} that $f_{j_0}$ extends to a holomorphic automorphism of $D$, since $D$ is simply connected. By the standard induction argument, each $f_j$ is a holomorphic automorphism of $D$ and thus $\lambda=m$.
\end{proof}

This result generalizes Mok's rigidity result when $D=\mathbb{B}^n$ \cite{M02}. 
Note that Mok's $p$-th root map provides a one-dimensional example from the unit disc $\Delta$ to the polydisc $\Delta^p$ that is holomorphically isometric but not totally geodesic \cite{M11}. 
Similar non-standard examples can be constructed from the polydisc $\Delta^n$ to its Cartesian product $(\Delta^n )^m$.
Nevertheless, 
we conjecture that the following rigidity result holds when the dimension is at least 2.

\begin{question}
Let $D$ be a smooth bounded domain in $\C^n, n \geq 2$ and $U \subset D$ be a connected open subset.  For $1\leq j\leq m$, let
$  f_j: U\rightarrow D$
be a holomorphic map, and assume that every \(f_j\) has full rank
\(n\) at some point of \(U\).  If there exists $\lambda>0$ such that
\begin{equation*}
  \sum_{j=1}^{m}f_j^*\omega_D=\lambda\omega_D
\end{equation*}
holds on $U$, 
does each $  f_j$ extend to an element in $\Aut(D)$?
\end{question}

\subsection{Analytic correspondence between complex analytic spaces}

\begin{definition}
Let \(D\subset\C^n\) be a bounded domain with Bergman metric $\omega_D$. A reduced
irreducible complex analytic space \(S\) of pure dimension \(n\) is called a complex analytic space associated to $D$
if
$ \pi: D\rightarrow S$
is a surjective holomorphic quotient map and additionally if there is a
nonempty dense open set $ S^\circ\subset\Reg S$
such that:
\begin{enumerate}
\item \( \pi \) is locally biholomorphic over \(S^\circ\);
\item a positive smooth \((1,1)\)-form \(\omega_S\) is defined on
  \(S^\circ\); 
\item $ \pi^*\omega_S=\omega_D$  on $\pi^{-1}(S^\circ)$.
\end{enumerate}
\end{definition}

\begin{definition}
A finite analytic correspondence on \(S\)
is a pure \(n\)-dimensional irreducible analytic subset
$ Y\subset S\times S$
such that the two restrictions of the coordinate projections
$ \pi_i:=\pr_i|_Y:Y\rightarrow S$
are proper with finite
fibers  and surjective, for $  i=1,2$. 
The generic degree of \(\pi_i\) is denoted by $ d_i:=\deg \pi_i.$
\end{definition}

When \(S\) and \(Y\) are algebraic, this is called a (finite) algebraic correspondence.

\begin{remark}
We regard \(Y\) as the multivalued map $ T_Y=\pi_1\circ \pi_2^{-1}.$
Thus \(\pi_2\) is the source projection, \(\pi_1\) is the target
projection, and there are generically \(d_2\) branches.  If the
opposite convention \(\pi_2\circ \pi_1^{-1}\) is chosen, all statements
remain true after interchanging \(d_1\) and \(d_2\).
\end{remark}

Choose a point of \(S^\circ\) outside the discriminants of both
projections and 
outside the union, over the finitely many branches, of the inverse images of
\(S\setminus S^\circ\) under the local branches.  On a sufficiently
small connected open subset  \(V\Subset S^\circ\), one
has $\pi_2^{-1}(V)=\bigsqcup_{j=1}^{d_2}V_j,$
where every $ \pi_2|_{V_j}:V_j\rightarrow V$
is biholomorphic and \(\pi_1(V_j)\subset S^\circ\).  Define
\begin{equation}\label{eq:theta}
 \theta_j := \pi_1 |_{V_j}\circ(\pi_2|_{V_j})^{-1} :V\rightarrow S^\circ.
\end{equation}
After shrinking \(V\), every \(\theta_j\) is a biholomorphism onto
its image.

\begin{definition}\label{def:preserve}
If $ \sum_{j=1}^{d_2}\theta_j^*\omega_S=\lambda \omega_S$
 for some $\lambda >0$, we call that the finite analytic correspondence 
 \(Y\) preserves $\omega_S$. 
\end{definition}

The proof of the next theorem follows from the argument in \cite{CU03}.

\begin{theorem}
\label{thm:reduction}
Let \(Y\subset S\times S\) be a finite analytic correspondence
preserving \(\omega_S\).
Then there are a connected open set
\(U\subset D\) and holomorphic maps $ f_j:U\rightarrow D$ for $1\leq j\leq d_2,$
such that
\begin{equation}\label{eq:sum-isometry}
 \sum_{j=1}^{d_2}f_j^*\omega_D=d_1\omega_D.
\end{equation}
\end{theorem}

\begin{proof}
Choose \(V\), \(V_j\), and \(\theta_j\) as in \eqref{eq:theta}.  Since \(\pi\) is locally
biholomorphic over \(S^\circ\), choose local inverse branches
\[ \sigma:V\rightarrow D,\quad \pi\circ\sigma=\operatorname{id}_V,\]
and
\[ \tau_j:\theta_j(V)\rightarrow D,\quad \pi \circ\tau_j=\operatorname{id}_{\theta_j(V)}.\]
Let $ U:=\sigma(V)$
and define
 $f_j := \tau_j\circ\theta_j\circ \pi |_U:U\rightarrow D.$
Then \(U\) is connected and open, every \(f_j\) is holomorphic, and
$ \pi \circ f_j=\theta_j\circ \pi |_U.$
It follows that 
\[
\begin{aligned}
 f_j^*\omega_D
 &=f_j^*(\pi^*\omega_S) =(\pi \circ f_j)^*\omega_S =(\theta_j\circ \pi |_U)^*\omega_S =(\pi |_U)^*(\theta_j^*\omega_S).
\end{aligned}
\]
Therefore \[
\begin{aligned}
 \sum_{j=1}^{d_2}f_j^*\omega_D =(\pi |_U)^*
   \left(\sum_{j=1}^{d_2}\theta_j^*\omega_S\right) =(\pi |_U)^*(d_1\omega_S) =d_1\omega_D.
\end{aligned}
\]
This proves \eqref{eq:sum-isometry}.
\end{proof}

\begin{definition}
If $f_j \in {\rm Aut}(D)$ for all $j=1, 2, \cdots, m$, then $Y$ is called a special correspondence.
\end{definition}

The following result follows by combining Theorem \ref{product}  and Theorem \ref{thm:reduction}.

\begin{corollary}
Let \(D\Subset\C^n\), \(n\geq2\), be a simply connected, bounded,
strongly pseudoconvex domain with real analytic boundary and $S$ be a complex analytic space associated to $D$. If a finite analytic correspondence $Y \subset S\times S$ preserves $\omega_S$,  
then $Y$ is a special correspondence.
\end{corollary}

Now, let $ \pi :D\rightarrow S=D/\Gamma$
is a regular holomorphic covering and $ \Gamma={\rm Deck}(\pi)\subset\Aut(D).$
The commensurator of \(\Gamma\) in \(\Aut(D)\) consists of all $ g\in\Aut(D)$ such that  $  \Gamma_g:=\Gamma\cap g^{-1}\Gamma g $ has finite index in both $\Gamma$ and $g^{-1}\Gamma g $, denoted by 
$ {\rm Comm}_{\Aut(D)}(\Gamma)$.

\begin{definition}
Let \(g\in  {\rm Comm}_{\Aut(D)}(\Gamma)\). 
 $ Y_g:=\Phi_g(D/\Gamma_g)$, the image of $ \Phi_g: D/\Gamma_g\rightarrow S\times S$, given by $\Phi_g([z]_{\Gamma_g})=(\pi(z), \pi(gz))$, is the modular correspondence induced by \(g\). 
An irreducible finite correspondence \(Y \subset S\times S \) is modular if $ Y=Y_g$
for some \(g\in {\rm Comm}_{\Aut(D)}(\Gamma)\). 
\end{definition}

By the standard deduction as in \cite{CU03}, we can easily obtain the following corollary.

\begin{corollary}
Let \(D\Subset\C^n\), \(n\geq2\), be a simply connected, bounded,
strongly pseudoconvex domain with real analytic boundary.
Let $ \Gamma={\rm Deck}(\pi)\subset\Aut(D)$ such that  $ \pi :D\rightarrow S=D/\Gamma$
is a regular holomorphic covering.
 If a finite analytic correspondence $Y \subset S\times S$ preserves $\omega_S$ up to the conformal constant $\lambda>0$, then $Y$ is a modular correspondence.
\end{corollary}


\section{On holomorphic maps preserving Bergman $(p, p)$-forms}

Clozel and Ullmo studied algebraic correspondences between Shimura
varieties that preserve volume forms \cite{CU03}. Their approach leads
naturally to the study of local holomorphic maps from a bounded
symmetric domain \(\Omega\subset\CC^n\) into its Cartesian product preserving the Bergman volume form. More
precisely, let \(U\subset\Omega\) be a connected open subset, and let
\[    F=(f_1,\ldots,f_m):U\rightarrow\Omega^m\]
be a holomorphic map satisfying
\[   \sum_{j=1}^m f_j^*\omega_\Omega^n =   \lambda\,\omega_\Omega^n\]
for some constant \(\lambda>0\). Mok and Ng proved that \(F\) is
totally geodesic \cite{MN12}. More generally, Mok posed the rigidity
problem for local holomorphic maps from \(\Omega\) into \(\Omega^m\)
that preserve Bergman \((p,p)\)-forms \cite{M11}. Combining the results
of \cite{Y17,DY26} yields an affirmative answer to this problem in all
cases. As a consequence, one obtains the corresponding modularity
theorem for algebraic correspondences between Shimura varieties that
preserve these \((p,p)\)-forms.

Chan and Yuan studied local holomorphic maps
$ f:(M_1,\omega_1)\rightarrow(M_2,\omega_2),$
where \(M_1\) and \(M_2\) are K\"ahler manifolds of complex dimensions
\(n\) and \(N\geq n\), respectively, satisfying
\[   f^*\omega_2^p=\lambda\,\omega_1^p\]
for some constant \(\lambda>0\). They proved that, when \(1 \leq p<n\), the
map \(f\) must be a holomorphic Bergman isometriy \cite{CY25}. Consequently,
the metric-preservation hypothesis in Theorem \ref{domain} - \ref{al3} may be
replaced by the assumption that \(f\) preserves the Bergman
\((p,p)\)-form, up to a positive constant, for any \(p<n\). 
In contrast, no
analogous result is currently known when \(p=n\). Moreover, it remains
open whether the conclusion of Theorem~\ref{product}, together with
its subsequent consequences, continues to hold under the analogous
preservation condition for Bergman \((p,p)\)-forms with \(p>1\).
These problems merit further investigation.

\section*{Acknowledgments}

The author is very grateful to Xiaojun Huang for helpful discussions 
during a satellite conference of ICM 2026 at Rutgers-New Brunswick from July 17 to July 21, which leads to Theorem \ref{main1}, Theorem \ref{thm:main2}, Theorem \ref{al2} and Theorem \ref{al3}.
Before posting this paper on arXiv, the author learned that Ebenfelt, Treuer and Xiao had independently
obtained some similar results in their preprint \cite{ETX26} (cf. Theorem  \ref{main1}, Theorem \ref{al2} and Proposition \ref{negligible} in the present paper). 
The author would also like to thank Ming Xiao for helpful discussion on Theorem \ref{main1} and Theorem \ref{al2}. 
The methods in two papers differ
substantially, where
Ebenfelt, Treuer and Xiao apply the
Calabi-type extension theorem while Irgens’ L2-envelope of holomorphy is used here. 
The author was partially supported by Zhejiang Provincial Natural Science Foundation of China
(Grant No. LQKWL26A0201).



\

\end{document}